\documentclass[11pt]{article}

\usepackage[margin=1.12in]{geometry}
\usepackage{amsmath,amssymb,amsthm,mathtools}
\usepackage{booktabs}
\usepackage{enumitem}
\usepackage{microtype}
\usepackage{hyperref}

\hypersetup{
  colorlinks=true,
  linkcolor=blue,
  citecolor=blue,
  urlcolor=blue
}

\newtheorem{theorem}{Theorem}[section]
\newtheorem{lemma}[theorem]{Lemma}
\newtheorem{corollary}[theorem]{Corollary}
\newtheorem{proposition}[theorem]{Proposition}

\theoremstyle{definition}
\newtheorem{definition}[theorem]{Definition}
\theoremstyle{remark}
\newtheorem{remark}[theorem]{Remark}

\newcommand{\F}{\mathbb F}
\newcommand{\Gr}{\operatorname{Gr}}
\newcommand{\Hom}{\operatorname{Hom}}
\newcommand{\Mat}{\operatorname{Mat}}
\newcommand{\rank}{\operatorname{rank}}
\newcommand{\cdim}{\operatorname{cdim}}

\newcommand{\Unif}{\operatorname{Unif}}
\newcommand{\Prob}{\mathbb P}

\newcommand{\eps}{\varepsilon}
\newcommand{\qbinom}[2]{\genfrac{[}{]}{0pt}{}{#1}{#2}_q}
\newcommand{\Pts}{\operatorname{Pts}}

\title{Default-Distance Entropy and Metric Dimension in Finite Geometries}
\author{Maximiliano Vazquez}
\date{}

\begin{document}
\maketitle

\begin{abstract}
A resolving set in a graph is a set of landmarks whose distance vectors
distinguish all vertices.  We use information theory to prove lower bounds
for metric dimension and class dimension in finite geometric
distance-regular graphs and association schemes.  The core idea is that, for
a fixed landmark, a random object usually lies in one overwhelmingly likely
distance or relation class.  For classical dual polar graphs, with rank and
type fixed and $q\to\infty$ through the admissible field orders, we prove
$\mu(\Gamma(q,d,e))=\Theta_{d,e}(q^e)$ for $d\geq2$ and $e>0$.  The lower bound uses opposition as the typical distance, while the upper bound takes all generators containing each of a constant number of \((d-1)\)-dimensional singular subspaces.  For Grassmann
graphs, bilinear forms graphs, and attenuated-space schemes, we obtain lower
bounds of the same exponential order as the known incidence constructions.
\end{abstract}

\section{Introduction}

Given a finite connected simple graph $G$, a set $S\subseteq V(G)$ 
resolves
$G$ if every vertex $x$ is uniquely determined by its vector of distances
\[
        x\longmapsto (d(s,x))_{s\in S}.
\]
The minimum size of such a set is the metric dimension of $G$, denoted 
here
by $\mu(G)$.\footnote{Some graph theory papers write $\beta(G)$, while
others use $\dim(G)$.  We reserve $\dim$ for vector space dimension.}
This language was introduced in graph theory by Slater under the name
locating sets and independently by Harary and Melter under the name metric
dimension \cite{Slater1975,HararyMelter1976}; the same idea had appeared
earlier in Blumenthal's work on distance geometry \cite{Blumenthal1953}.
The interpretation in terms of landmarks was emphasized by Khuller,
Raghavachari, and Rosenfeld \cite{KhullerRaghavachariRosenfeld}.

The most immediate lower bound is a counting bound.  If $G$ has diameter
$D$, then one landmark returns one of $D+1$ possible distances, and every
resolving set satisfies
\[
        |V(G)|\leq (D+1)^{|S|},
        \qquad
        \mu(G)\geq \frac{\log |V(G)|}{\log(D+1)}.
\]
This estimate treats the possible distance values as if they were equally
informative.  In many distance-regular graphs, however, almost every 
vertex
has the same distance from a fixed landmark.  If $X$ is uniform on $V(G)$
and $S$ resolves $G$, then its distance code is injective, so by entropy subadditivity,
\[
        \log |V(G)|=H(X)
        \leq \sum_{s\in S}H(d(s,X)).
\]
When $G$ is vertex-transitive, the summands have the same distribution.  
If
one distance occurs with probability $1-\eps$, then one coordinate carries
only
\[
        \eps\log(1/\eps)+O(\eps)
\]
bits of entropy.  This is the common mechanism behind the lower bounds in
the paper.

This information-theoretic viewpoint goes back to the coin-weighing work 
of
Erd\H{o}s--R\'enyi and Pippenger
\cite{ErdosRenyi1963,Pippenger1977}.  More recently, D\'{\i}az, Hartle, and
Moore used the same entropy subadditivity principle for sparse
Erd\H{o}s--R\'enyi graphs \cite{DiazHartleMoore2026}.  The entropy inequality
itself is standard, however appears to be somewhat underutilized in the metric dimension literature.  Our contribution is to identify and count the first
exceptional strata in several finite geometric families.  For dual polar
graphs, we also pair the entropy lower bound with a new upper construction
using panel pencils.

The graph families considered here belong to the classical theory of
distance-regular graphs; see Brouwer--Cohen--Neumaier \cite{BCN} and the
survey of van Dam--Koolen--Tanaka \cite{VanDamKoolenTanaka}.  Metric
dimension entered this setting systematically through Bailey and Cameron's
survey on bases, metric dimension, coherent configurations, and 
association
schemes \cite{BaileyCameron}.  In that language, resolving sets for 
coherent
configurations go back to Babai's distinguishing sets
\cite{Babai1981,BaileyCameron}.

Most of the subsequent finite geometric literature is constructive.
Bailey and Meagher used point--subspace incidence for Grassmann graphs
\cite{BaileyMeagher}; Feng and Wang treated bilinear forms graphs
\cite{FengWang}, with the remaining parameter cases completed by Guo, Li,
and Wang \cite{GuoLiWangFour}; Guo, Li, and Wang treated attenuated-space
schemes \cite{GuoLiWang}; and Bailey and Spiga bounded dual polar graphs by the 
real
rank of the point--generator incidence matrix \cite{BaileySpiga}.  These
constructions provide effective upper bounds but leave the corresponding
lower bound theory much less developed.

The strongest result of the present paper is for dual polar graphs.  
Bailey
and Spiga explicitly observed that, beyond the generalized quadrangle 
case,
no general lower bound was known for dual polar graphs of arbitrary
diameter \cite{BaileySpiga}.  For every fixed rank $d$ and every classical
parameter $e>0$, we prove
\[
        \mu(\Gamma(q,d,e))=\Theta_{d,e}(q^e).
\]
The lower bound comes from the fact that a random generator is opposite a
fixed generator with probability $1-q^{-e}+o(q^{-e})$.  The matching-order
upper bound uses the line geometry of the dual polar space: one samples a
constant number of panels and takes all $q^e+1$ generators in each
corresponding pencil.  These correlated bundles remove a logarithmic 
factor
that arises from naive independent sampling.  The result answers the 
lower bound
question of Bailey and Spiga in the fixed-rank regime and improves their
general incidence-rank upper bound by a polynomial factor. 

The remaining applications show that the same entropy principle is not
specific to polar spaces.  For Grassmann graphs, bilinear forms graphs, 
and
attenuated spaces, we obtain matching exponential orders against the known
incidence constructions.  The attenuated-space argument is fiberwise and
combines the Grassmann intersection obstruction with a rank-defect
obstruction from bilinear forms graphs.

\subsection*{Main results}

For the dual polar row below, set
\[
        A_{d,e}=de+\binom d2,\qquad
        k_{d,e}=\left\lfloor\frac{2A_{d,e}}{\min\{1,e\}}\right\rfloor+1.
\]
All asymptotics are in the indicated regimes, and suppressed constants
depend only on the fixed parameters.
\begin{center}
\small
\begin{tabular}{lll}
\toprule
Family & Regime & Bounds \\
\midrule
$G_q(n,k)$ &
$\substack{\text{fixed }q,k\\ n\to\infty}$ &
$\left(\dfrac{k(q-1)}{(q^k-1)^2}+o(1)\right)q^n
 \leq \mu(G_q(n,k))\leq \qbinom n1$ \cite{BaileyMeagher} \\
\addlinespace
$H_q(n,d)$ &
$\substack{\text{fixed }q,d\\ n\to\infty}$ &
$\left(\dfrac{d(q-1)}{q^d-1}+o(1)\right)q^n
 \leq \mu(H_q(n,d))\leq q^{n+d-1}$ \cite{FengWang,GuoLiWangFour} \\
\addlinespace
$\Omega_m$ &
$\substack{\text{fixed }q,m\geq2\\ n,\ell\to\infty}$ &
$\left(\dfrac{m(q-1)}{(q^m-1)^2}+o(1)\right)q^{n+\ell}
 \leq \cdim(\Omega_m)\leq q^\ell\qbinom n1$ \cite{GuoLiWang} \\
\addlinespace
$\Omega_1$ &
$n,\ell\geq1$ &
$\cdim(\Omega_1)=(q^\ell-1)\qbinom n1$ \\
\addlinespace
$\Gamma(q,d,e)$ &
$\substack{\text{fixed }d,\text{ type},e>0\\ q\to\infty}$ &
$\left(\dfrac{A_{d,e}}{e}+o(1)\right)q^e
 \leq \mu(\Gamma(q,d,e))\leq k_{d,e}(q^e+1)$ \\
\bottomrule
\end{tabular}
\end{center}

Thus all four families are determined up to constant factors in the stated
regimes.  The dual polar result is the sharpest qualitative advance: to our
knowledge, it supplies the first general lower bound for arbitrary fixed rank,
replaces the previous
polynomially larger incidence upper bounds by $O_{d,e}(q^e)$, and leaves 
the
leading constant as the main unresolved issue.  The $m=1$ row is the exact
formula proved in Proposition~\ref{obs:rank-one}.

\subsection*{Organization}

Section~\ref{sec:entropy} gives the entropy principle for distance and
relation codes.  Section~\ref{sec:grassmann} isolates the intersection
obstruction in Grassmann graphs, and Section~\ref{sec:bilinear} isolates the
rank-defect obstruction in bilinear forms graphs.  Section~\ref{sec:attenuated}
combines the two in the graph-of-a-map model.  Section~\ref{sec:dualpolar}
proves the dual-polar lower bound and then builds the matching order upper
bound from panel pencils.

\section{Entropy of distance and relation codes}\label{sec:entropy}

All logarithms are base \(2\), except that \(\ln\) denotes the natural 
logarithm.

Let \(G\) be a finite connected graph.  For \(S\subseteq V(G)\), write
\[
        \Phi_S(x)=(d(s,x))_{s\in S}
\]
for the distance code of \(x\).  The set \(S\) resolves a set of vertices 
\(T\) precisely when \(\Phi_S\) is injective on \(T\).  If \(X\) is a 
random vertex supported on \(T\), then an injective code determines \(X\), 
and hence
\[
        H(X)=H(\Phi_S(X))
        \leq
        \sum_{s\in S} H(d(s,X)).
\]

\begin{theorem}\label{thm:vt-entropy}
Let \(G\) be a finite connected vertex-transitive graph, let 
\(X\sim\Unif(V(G))\), and fix \(s_0\in V(G)\).  Then
\[
        \mu(G)
        \geq
        \frac{\log |V(G)|}{H(d(s_0,X))},
\]
provided the denominator is nonzero.
\end{theorem}

\begin{proof}
For any resolving set \(S\), the preceding inequality gives
\[
        \log |V(G)| \leq \sum_{s\in S} H(d(s,X)).
\]
If \(g(s_0)=s\), then \(d(s,X)=d(g(s_0),X)=d(s_0,g^{-1}X)\), and 
\(g^{-1}X\) is again uniform.  Thus each summand equals \(H(d(s_0,X))\), 
and the result follows after minimizing over \(S\).
\end{proof}

We use the following convention for association schemes.  A symmetric 
association scheme on a finite set \(\Omega\) is a partition
\[
        \Omega\times\Omega=R_0\sqcup R_1\sqcup\cdots\sqcup R_D
\]
with \(R_0=\{(x,x):x\in\Omega\}\), each \(R_i\) symmetric, and constants 
\(p_{ij}^k\) such that for every \((x,y)\in R_k\),
\[
        p_{ij}^k
        =
        |\{z\in\Omega:(x,z)\in R_i,\ (z,y)\in R_j\}|.
\]
Some authors use ``association scheme'' for the more general, not 
necessarily symmetric or commutative, object, and others build 
commutativity into the definition.  Here all schemes are symmetric, hence 
commutative, and the relation index is an unordered invariant of a pair.  

\begin{definition}\label{def:classdimension}
Let \((\Omega,\{R_0,\ldots,R_D\})\) be a symmetric association scheme.  
For \(s,x\in\Omega\), let \(i(s,x)\) be the unique index such that 
\((s,x)\in R_{i(s,x)}\).  A set \(S\subseteq\Omega\) is 
\emph{class-resolving} if
\[
        x\longmapsto (i(s,x))_{s\in S}
\]
is injective.  The least size of such an \(S\) is the \emph{class 
dimension}, denoted \(\cdim(\Omega)\).  When the scheme is the distance 
scheme of a distance-regular graph, this is ordinary metric dimension.
\end{definition}

The index set need not literally be \(\{0,\ldots,D\}\).  In later
examples, a pair such as \((i,j)\) is regarded as a single relation label.

The same entropy subadditivity applies to relation codes. If \(S\) is
class-resolving and \(X\) is any random element of a subset 
\(T\subseteq\Omega\)
on which the relation code is injective, then
\[
H(X)\le \sum_{s\in S} H(i(s,X)).
\]
We will use this localized form below, rather than a global transitivity
bound, because the attenuated-space argument is carried out fiberwise.

The next lemma is the only entropy estimate used below.  It is really just 
Fano's inequality in the case of a constant estimator, and we state it 
here to avoid writing out the same argument in the theorems below.

\begin{lemma}\label{lem:default}
Let \(R\) be a random variable taking values in a finite alphabet \(A\), 
with
\(|A|=M\geq2\).  Fix \(a\in A\), and set \(\varepsilon=\Pr(R\neq a).\)
Then
\[
        H(R)\leq H_2(\varepsilon)+\varepsilon\log(M-1),
\]
where
\[
        H_2(t)=-t\log t-(1-t)\log(1-t)
\]
is the binary entropy function, with \(0\log0=0\).
Moreover, if \(\varepsilon\leq\delta\leq1/2\), then
\[
        H(R)\leq \delta\log(1/\delta)+O_M(\delta).
\]
\end{lemma}

\begin{proof}
Let \(E=\mathbf 1_{\{R\neq a\}}\). \(E\) is determined by \(R\), so the 
chain rule gives \( H(R)=H(E)+H(R\mid E)\). On \(E=0\), the value of \(R\) 
is fixed.  On \(E=1\), the value of \(R\) lies in the alphabet 
\(A\setminus\{a\}\), which has size \(M-1\).  Hence
\[
        H(R\mid E)
        =
        \varepsilon H(R\mid E=1)
        \leq
        \varepsilon\log(M-1),
\]
and \(H(E)=H_2(\varepsilon)\).  This proves the first inequality.

For the last assertion, use monotonicity of \(H_2\) on \([0,1/2]\) and the 
expansion
\[
        H_2(\delta)=\delta\log(1/\delta)+O(\delta)
        \qquad(\delta\downarrow0).
\]
The term \(\delta\log(M-1)\) is absorbed into \(O_M(\delta)\).
\end{proof}

\section{Grassmann graphs}\label{sec:grassmann}

Throughout the finite field applications, \(q\) is a prime power.  We
write \(\qbinom ab\) for a Gaussian coefficient and adopt the convention
that \(\qbinom ab=0\) when \(b<0\) or \(b>a\).

The Grassmann graph is the finite field analogue of the Johnson graph.  In 
the Johnson graph \(J(n,k)\), vertices are \(k\)-subsets of an \(n\)-set, 
and distance records the size of the intersection.  In \(G_q(n,k)\), 
vertices are \(k\)-dimensional subspaces of \(\F_q^n\), and intersection 
size is replaced by intersection dimension.  Bailey and Meagher initiated 
the metric dimension study of Grassmann graphs and proved the incidence 
upper bound used below \cite{BaileyMeagher}.

Throughout this section we assume \(1\leq k\leq n/2\).  This is automatic
for all sufficiently large \(n\) in the fixed-\(k\) regime.

The vertex set of \(G_q(n,k)\) is \(\Gr_q(n,k)\), the set of 
\(k\)-dimensional subspaces of \(\F_q^n\).  Two vertices are adjacent when 
they meet in dimension \(k-1\).  The distance formula is
\[
        d(U,W)=k-\dim(U\cap W).
\]
Thus a vertex at maximum distance from \(U\) is simply a \(k\)-subspace 
disjoint from \(U\).  The entropy calculation below says that a fixed 
landmark sees a random \(k\)-subspace at this maximum distance with 
overwhelming probability.

\begin{theorem}\label{thm:grassmann}
For fixed \(q\) and \(k\geq1\),
\[
        \mu(G_q(n,k))
        \geq
        \left(\frac{k(q-1)}{(q^k-1)^2}+o(1)\right)q^n
        \qquad (n\to\infty).
\]
Consequently,
\[
        \mu(G_q(n,k))=\Theta_{q,k}(q^n).
\]
More precisely,
\[
        \mu(G_q(n,k))
        \leq
        \qbinom n1
        =
        \left(\frac1{q-1}+o(1)\right)q^n.
\]
\end{theorem}

\begin{proof}
For \(k\geq2\), the upper bound is the point--subspace incidence
construction of Bailey and Meagher \cite[Theorem~5]{BaileyMeagher}; for
\(k=1\), the graph is complete and the displayed upper bound is immediate.
We prove the lower bound.

The action of \(\operatorname{GL}_n(q)\) on \(\Gr_q(n,k)\) is transitive,
so Theorem~\ref{thm:vt-entropy} applies.

Fix a \(k\)-subspace \(U\leq\F_q^n\), and let \(X\) be uniform on 
\(\Gr_q(n,k)\).  The default distance from \(U\) is \(k\), and the 
exceptional event is \(U\cap X\neq0\).  If \(U\cap X\neq0\), then \(U\cap 
X\) contains at least one one-dimensional subspace \(L\leq U\).  Therefore
\[
        \{U\cap X\neq0\}
        \subseteq
        \bigcup_{\substack{L\leq U\\ \dim L=1}}\{L\leq X\}.
\]
For a fixed one-dimensional subspace \(L\leq U\),
\[
        \Prob(L\leq X)
        =
        \frac{\qbinom{n-1}{k-1}}{\qbinom nk}
        =
        \frac{q^k-1}{q^n-1},
\]
where we used
\[
        \qbinom nk
        =
        \frac{q^n-1}{q^k-1}\qbinom{n-1}{k-1}.
\]
Since \(U\) contains \(\qbinom k1=(q^k-1)/(q-1)\) one-dimensional
subspaces, the union bound 
gives
\[
        \eps_n:=\Prob(U\cap X\neq0)
        \leq
        \delta_n:=\qbinom k1\frac{q^k-1}{q^n-1}
        =
        \left(\frac{(q^k-1)^2}{q-1}+o(1)\right)q^{-n}.
\]
Set
\[
        C_{q,k}=\frac{(q^k-1)^2}{q-1}.
\]
Then \(\delta_n=(C_{q,k}+o(1))q^{-n}\), so
\[
        \log(1/\delta_n)=n\log q+O_{q,k}(1).
\]
In particular, \(\delta_n\leq1/2\) for all sufficiently large \(n\).
The response \(d(U,X)\) has the fixed alphabet \(\{0,1,\ldots,k\}\).  
Lemma~\ref{lem:default} therefore gives
\[
        H(d(U,X))
        \leq
        \delta_n\log(1/\delta_n)+O_{q,k}(\delta_n)
        =
        (C_{q,k}+o(1))nq^{-n}\log q.
\]
Also
\[
        \log |\Gr_q(n,k)|
        =
        \log\qbinom nk
        =
        k(n-k)\log q+O_{q,k}(1)
        =
        kn\log q\,(1+o(1)).
\]
Substituting these estimates into Theorem~\ref{thm:vt-entropy} gives
\[
        \mu(G_q(n,k))
        \geq
        \frac{kn\log q\,(1+o(1))}
        {\left(\frac{(q^k-1)^2}{q-1}+o(1)\right)nq^{-n}\log q}
        =
        \left(\frac{k(q-1)}{(q^k-1)^2}+o(1)\right)q^n.
\]
\end{proof}

\section{Bilinear forms graphs}\label{sec:bilinear}

Grassmann graphs are governed by whether two subspaces intersect.  Bilinear
forms graphs give the parallel rank-defect case.  They are another classical
family of distance-regular graphs and may be viewed as rank-metric analogues of Hamming graphs:
the ambient set is a vector space of matrices, and the distance between 
two matrices is the rank of their difference rather than the number of 
coordinates in which they differ.  In the language of association schemes, 
this is the bilinear forms scheme, one of the standard finite-geometric 
schemes in the classical list \cite[Ch.~9]{BCN}.  Feng and Wang studied 
its metric dimension by constructing resolving sets from incidence data 
\cite{FengWang}.

Let \(H_q(n,d)\) have vertex set \(\Mat_{n\times d}(\F_q)\), with
\[
        d(A,B)=\rank(A-B).
\]
For \(n\geq d\), the generic distance from a fixed landmark is \(d\), 
because a random \(n\times d\) matrix usually has full column rank. 

\begin{theorem}\label{thm:bilinear}
For fixed \(q\) and \(d\geq1\),
\[
        \mu(H_q(n,d))
        \geq
        \left(\frac{d(q-1)}{q^d-1}+o(1)\right)q^n
        \qquad (n\to\infty).
\]
Feng and Wang prove \(\mu(H_q(n,d))\le q^{n+d-1}\) for \(n\ge d+2\)
\cite[Theorem~1.1]{FengWang}, and Guo, Li, and Wang extend the same bound to
\(n=d,d+1\) \cite{GuoLiWangFour}.  Thus the bound holds for \(n\geq d\), so
\(\mu(H_q(n,d))=\Theta_{q,d}(q^n)\).  The leading constants are still far
apart: the upper construction has constant \(q^{d-1}\), whereas the 
entropy lower constant is \(d(q-1)/(q^d-1)\).
\end{theorem}

\begin{proof}
Matrix translations \(X\mapsto X+C\) are graph automorphisms and act
transitively on the vertices, so Theorem~\ref{thm:vt-entropy} applies.

Let \(X\) be a uniform \(n\times d\) matrix.  For a fixed landmark \(A\), 
the map \(X\mapsto X-A\) is a bijection of \(\Mat_{n\times d}(\F_q)\).  
Hence \(X-A\) is uniform whenever \(X\) is uniform, and the distribution 
of \(d(A,X)=\rank(X-A)\) is the distribution of \(R_n=\rank X\), 
independent of \(A\).  The default rank is \(d\), and the exceptional 
probability is
\[
        p_n=\Prob(R_n<d).
\]
The columns of \(X\) are linearly independent exactly when the \(i\)-th 
new column avoids the span of the previous \(i\) columns for 
\(i=0,\ldots,d-1\).  Therefore
\[
        \Prob(R_n=d)=\prod_{i=0}^{d-1}(1-q^{i-n}).
\]
Expanding the finite product,
\[
        p_n
        =
        1-\prod_{i=0}^{d-1}(1-q^{i-n})
        =
        \sum_{i=0}^{d-1}q^{i-n}+O_{q,d}\left(\sum_{0\leq i<j\leq 
d-1}q^{i+j-2n}\right)
        =
        \left(\frac{q^d-1}{q-1}+o(1)\right)q^{-n}.
\]
Set
\[
        C_{q,d}=\frac{q^d-1}{q-1}.
\]
Then \(p_n=(C_{q,d}+o(1))q^{-n}\), and hence
\[
        \log(1/p_n)=n\log q+O_{q,d}(1).
\]
In particular, \(p_n\leq1/2\) for all sufficiently large \(n\).
Since the rank response has the fixed alphabet \(\{0,1,\ldots,d\}\), 
Lemma~\ref{lem:default} gives
\[
        H(R_n)
        \leq
        p_n\log(1/p_n)+O_{q,d}(p_n)
        =
        (C_{q,d}+o(1))nq^{-n}\log q.
\]
Since \(|V(H_q(n,d))|=q^{nd}\), Theorem~\ref{thm:vt-entropy} gives
\[
        \mu(H_q(n,d))
        \geq
        \frac{nd\log q}
        {\left(\frac{q^d-1}{q-1}+o(1)\right)nq^{-n}\log q}
        =
        \left(\frac{d(q-1)}{q^d-1}+o(1)\right)q^n.
\]
\end{proof}

\begin{remark}
The attenuated-space argument below also recovers this lower bound, but
keeping the calculation here isolates the rank-defect mechanism.  At the
boundary \(m=n\), there is a single base
subspace \(W\), and \(\Omega_n\) identifies with the rank-metric space
\(\Hom(W,E)\).  In the asymptotic regime \(\ell\to\infty\) with \(m\)
fixed, this is \(H_q(\ell,m)\).  The bilinear proof is the cleanest instance
of the rank-defect 
obstruction, and we will reuse the same calculation in
Section~\ref{sec:attenuated}.
\end{remark}

\section{Attenuated-space association schemes}\label{sec:attenuated}

The last two sections isolated an intersection obstruction and a rank-defect
obstruction.  Attenuated spaces combine them.  Since these spaces are less
familiar than Grassmann or polar spaces, we describe the model in some detail.
Let \(q\) be a prime power, let
\(n\geq1\) and \(\ell\geq0\), and fix \(m\) with \(1\leq m\leq n\).
Let \(V\) be an \((n+\ell)\)-dimensional vector space over \(\F_q\), and
fix an \(\ell\)-dimensional subspace \(E\leq V\).  The attenuated space 
associated to \((V,E)\) consists of all subspaces of \(V\) that meet \(E\) 
trivially.  Put
\[
        \Omega_m=
        \{X\leq V:\dim X=m,\ X\cap E=0\}.
\]
The set \(\Omega_m\) carries a symmetric association scheme, usually 
called the association scheme based on an attenuated space.  Wang, Guo, 
and Li introduced these schemes and computed their intersection numbers, 
and later work computed character tables and automorphism groups 
\cite{WangGuoLiAttenuated,Kurihara,LiuWangAttenuated}.  Guo, Li, and Wang used 
incidence matrices in finite attenuated spaces to obtain upper bounds for 
class dimension \cite{GuoLiWang}.  Our lower bound is the entropy 
counterpart to that incidence construction.

Choose a complement \(W\) to \(E\), so that \(V=W\oplus E\) and 
\(W\cong\F_q^n\).  Let \(\pi:V\to W\) be projection along \(E\).  If 
\(X\in\Omega_m\), then \(\pi|_X\) is injective because \(X\cap E=0\), and 
hence \(\pi(X)\) is an \(m\)-subspace of \(W\).  Conversely, once an 
\(m\)-subspace \(U\leq W\) is fixed, every linear map \(f:U\to E\) gives 
an element of \(\Omega_m\) through its graph
\[
        \operatorname{graph}(f)=\{u+f(u):u\in U\}\leq W\oplus E.
\]
Thus an attenuated \(m\)-subspace is equivalently the data
\[
        (U,f),
        \qquad
        U\in\Gr_q(n,m),\quad f\in\Hom(U,E).
\]
The fiber over \(U\) is
\[
        F_U=\{\operatorname{graph}(f): f\in\Hom(U,E)\},
        \qquad |F_U|=q^{m\ell},
\]
and therefore
\[
        |\Omega_m|=q^{m\ell}\qbinom nm.
\]

The relation between two elements is transparent in the graph-of-a-map
model.  Write
\[
        X=\operatorname{graph}(f:U\to E),
        \qquad
        Y=\operatorname{graph}(h:T\to E).
\]
It is determined by \(\dim(U\cap T)\) and \(\dim(X\cap Y)\).  In the
rank-defect indexing used here, the relation \(R_{(i,j)}\) is characterized
by
\[
        \dim(U\cap T)=m-i,
        \qquad
        \dim(X\cap Y)=m-i-j.
\]
The valid indices satisfy
\[
        0\leq i\leq\min\{m,n-m\},
        \qquad
        0\leq j\leq\min\{m-i,\ell\}.
\]
We will not need the intersection numbers themselves, but we do need to 
see
how the maps affect the second dimension.  Put \(A=U\cap T\).  Then
\[
        X\cap Y
        =
        \{u+f(u):u\in A,\ f(u)=h(u)\},
\]
and rank--nullity gives
\[
        \dim(X\cap Y)
        =
        \dim A-\rank\bigl((f|_A)-(h|_A)\bigr).
\]

Recall that for \(s,X\in\Omega_m\), we write \(i(s, X)\) for the relation 
index of the pair \((s,X)\). If \(X\) is random, then \(i(s,X)\) is the 
response random variable contributed by the landmark \(s\). What matters 
for our proof is that the relation index is determined by \(\dim(U \cap 
T)\) and \(\rank((f|_A)-(h|_A))\).

The notation used below can be summarized as follows.
\begin{center}
\begin{tabular}{ll}
\toprule
Notation & Meaning \\
\midrule
$U,T$ & the projected $m$-subspaces \\
$f,h$ & the maps whose graphs are $X,Y$ \\
$A=U\cap T$ & the common domain \\
$r=\dim A$ & the dimension of that domain \\
$\rank((f|_A)-(h|_A))$ & the rank part of the relation \\
\bottomrule
\end{tabular}
\end{center}

The two boundary cases are worth recording.  If \(\ell=0\), then \(E=0\), 
every fiber has one point, and \(\Omega_m=\Gr_q(n,m)\).  The attenuated 
scheme is the Grassmann scheme.  If \(n=m\), then every projected subspace 
\(U\) equals \(W\), so \(\Omega_n\) identifies with \(\Hom(W,E)\).  After
choosing bases, its elements are \(\ell\times n\) matrices, and the relation
between two elements is the rank of their difference.  Thus this boundary
is \(H_q(\ell,n)\) when \(\ell\geq n\), or equivalently \(H_q(n,\ell)\)
after transposition when \(\ell<n\), and the fiber theorem below recovers
its entropy lower bound.  It does not
recover the Grassmann lower bound because for 
\(\ell=0\) the fibers have entropy zero, so the fiber inequality becomes 
vacuous.  This is why the separate Grassmann section is mathematically 
useful rather than merely expository.

For \(0\leq r\leq m\), write
\[
        \mathcal H_r=H(\rank M_{\ell\times r}),
\]
where \(M_{\ell\times r}\) is a uniform \(\ell\times r\) matrix over 
\(\F_q\), and put \(\mathcal H_0=0\).  For a fixed \(m\)-subspace \(T\leq 
W\) and a random \(U\in\Gr_q(n,m)\), define
\[
        p_r=\Prob(\dim(U\cap T)=r).
\]
The standard finite-field subspace count gives, for fixed 
\(T\in\Gr_q(n,m)\) and random \(U\in\Gr_q(n,m)\),
\[
        p_r=
        \frac{q^{(m-r)^2}\qbinom mr\qbinom{n-m}{m-r}}{\qbinom nm}.
\]
Indeed, the numerator is the number of \(m\)-subspaces \(U\leq W\) 
satisfying \(\dim(U\cap T)=r\); see \cite[Corollary~A.2]{Kurihara}.

\begin{theorem}\label{thm:attenuated-fiber}
If \(\sum_{r=1}^m p_r\mathcal H_r>0\), then the attenuated-space 
association scheme on \(\Omega_m\) satisfies
\[
        \cdim(\Omega_m)
        \geq
        \frac{m\ell\log q}{\sum_{r=1}^m p_r\mathcal H_r}.
\]
\end{theorem}

For \(\ell\geq1\), the denominator is automatically positive: \(p_m>0\),
and the rank of a uniform \(\ell\times m\) matrix is not deterministic.
The positivity condition only excludes the degenerate boundary \(\ell=0\).

\begin{proof}
Let \(S\) be class-resolving, and fix a base subspace \(U\in\Gr_q(n,m)\).  
Since \(S\)
resolves all of \(\Omega_m\), its response map is injective on the fiber
\(F_U\).  For \(s\in S\), write \(T=\pi(s)\) for its base subspace and 
write
\(s=\operatorname{graph}(h:T\to E)\).  Set \(A=U\cap T\) and \(r=\dim A\).

If \(X\in F_U\), then \(X=\operatorname{graph}(f:U\to E)\) for a unique
\(f\in\Hom(U,E)\).  As previously discussed we have
\[
        \dim(s\cap X)
        =
        r-\rank\bigl((f|_A)-(h|_A)\bigr).
\]
With \(U\), \(T\), and \(h\) fixed, the first part of the relation index 
is
fixed by \(r=\dim(U\cap T)\), and the second part is therefore in 
bijection with
\[
        \rank\bigl((f|_A)-(h|_A)\bigr).
\]

Now choose \(X\) uniformly from \(F_U\).  Equivalently, \(f\) is uniform 
in
\(\Hom(U,E)\).  The restriction map \(\Hom(U,E)\to\Hom(A,E)\) is 
surjective
with equal-sized fibers, so \(f|_A\) is uniform in \(\Hom(A,E)\).  
Translating by
the fixed map \(-h|_A\) preserves uniformity.  Hence
\((f|_A)-(h|_A)\) is a uniform element of \(\Hom(A,E)\), whose rank has 
the same
distribution as the rank of a uniform \(\ell\times r\) matrix over 
\(\F_q\).
Therefore the response entropy of \(s\) on \(F_U\) is
\[
        H(i(s,X))
        =
        \mathcal H_r
        =
        \mathcal H_{\dim(U\cap \pi(s))}.
\]
This includes the case \(r=0\), where the response is constant and
\(\mathcal H_0=0\).

Applying the entropy principle inside the fiber gives
\[
        m\ell\log q=\log |F_U|
        \leq
        \sum_{s\in S} \mathcal H_{\dim(U\cap \pi(s))}.
\]
Summing over all base subspaces \(U\in\Gr_q(n,m)\) and interchanging sums 
gives  
\[
        \qbinom{n}{m}m\ell\log q
        \leq
        \sum_{U\in\Gr_q(n,m)}
        \sum_{s\in S}
        \mathcal H_{\dim(U\cap \pi(s))}
        =
        \sum_{s\in S}
        \sum_{U\in\Gr_q(n,m)}
        \mathcal H_{\dim(U\cap \pi(s))}.
\]
The group \(\operatorname{GL}_n(q)\) is transitive on \(\Gr_q(n,m)\), so 
the inner sum is independent of the particular base subspace \(\pi(s)\).  
For any fixed \(T\in\Gr_q(n,m)\),
\[
        \sum_{U\in\Gr_q(n,m)}
        \mathcal H_{\dim(U\cap T)}
        =
        \qbinom{n}{m}
        \sum_{r=1}^m p_r\mathcal H_r,
\]
where the \(r=0\) term vanishes because \(\mathcal H_0=0\).  Hence
\[
        \qbinom{n}{m}m\ell\log q
        \leq
        |S|\qbinom{n}{m}
        \sum_{r=1}^m p_r\mathcal H_r.
\]
Canceling \(\qbinom{n}{m}\) and using the assumption that the denominator 
is positive proves the theorem.
\end{proof}

\begin{corollary}\label{cor:attenuated-fixedrank}
Fix \(q\) and \(m\geq1\).  Along any joint limit with
\(n\to\infty\) and \(\ell\to\infty\),
\[
        \cdim(\Omega_m)
        \geq
        \left(\frac{m(q-1)}{(q^m-1)^2}+o(1)\right)q^{n+\ell}.
\]
Together with the incidence-rank upper bound of Guo, Li, and Wang for \(m\geq2\)
and Proposition~\ref{obs:rank-one} for \(m=1\), this gives
\[
        \cdim(\Omega_m)=\Theta_{q,m}(q^{n+\ell}).
\]
\end{corollary}

\begin{proof}
By Theorem~\ref{thm:attenuated-fiber}, it remains to estimate the averaged
fiber entropy \(D_{n,\ell}:=\sum_{r=1}^m p_r\mathcal H_r\).
We first estimate the Grassmann intersection probabilities \(p_r\).  In 
the
fixed-\(q,m\), \(n\to\infty\) regime, the dominant nonzero contribution 
comes
from the stratum \(\dim(U\cap T)=1\):
\[
        p_1=
        \left(\frac{(q^m-1)^2}{q-1}+o(1)\right)q^{-n}.
\]
The remaining strata are lower order.  Indeed, if \(\dim(U\cap T)\geq 2\), 
then
\(U\cap T\) contains a two-dimensional subspace \(M\leq T\).  A union 
bound over
the \(\qbinom m2\) choices of \(M\) gives
\[
        \sum_{r=2}^m p_r
        \leq
        \qbinom m2
        \frac{\qbinom{n-2}{m-2}}{\qbinom nm}
        =
        O_{q,m}(q^{-2n}).
\]

We next estimate the rank entropies.  A uniform \(\ell\times 1\) matrix 
has rank
zero with probability \(q^{-\ell}\) and rank one otherwise, so
\[
        \mathcal H_1
        =
        H_2(q^{-\ell})
        =
        \ell q^{-\ell}\log q\,(1+o(1)).
\]
For \(1\leq r\leq m\), a uniform \(\ell\times r\) matrix has default rank 
\(r\).  Since \(r\leq m\) and \(m\) is fixed, we have \(\ell\geq r\) for
all sufficiently large \(\ell\),
and its probability of rank defect is \(O_{q,m}(q^{-\ell})\), as was shown 
in Theorem~\ref{thm:bilinear}.  Since the rank
defect probability is then at most \(1/2\), and the rank alphabet has size
at most \(m+1\), Lemma~\ref{lem:default} gives the 
uniform
bound
\[
        \mathcal H_r
        =
        O_{q,m}(\ell q^{-\ell}\log q).
\]
Therefore
\[
        D_{n,\ell}
        =
        p_1\mathcal H_1
        +
        \sum_{r=2}^m p_r\mathcal H_r
        =
        \left(\frac{(q^m-1)^2}{q-1}+o(1)\right)
        \ell q^{-n-\ell}\log q.
\]
Substituting this estimate into Theorem~\ref{thm:attenuated-fiber} gives
\[
        \cdim(\Omega_m)
        \geq
        \frac{m\ell\log q}{D_{n,\ell}}
        =
        \left(\frac{m(q-1)}{(q^m-1)^2}+o(1)\right)q^{n+\ell}.
\]
For \(2\leq m\leq n-1\), the incidence-rank construction of Guo, Li, and
Wang gives
\(\cdim(\Omega_m)\leq q^\ell\qbinom n1=O_q(q^{n+\ell})\).  For fixed
\(m\geq2\), this range holds for all sufficiently large \(n\).  When
\(m=1\), Proposition~\ref{obs:rank-one} gives the stronger exact bound
needed for the same conclusion.  Thus the stated
\(\Theta_{q,m}\)-bound follows.
\end{proof}

\begin{proposition}\label{obs:rank-one}
For all prime powers \(q\) and all \(n,\ell\geq 1\),
\[
        \cdim(\Omega_1)
        =
        (q^\ell-1)\frac{q^n-1}{q-1}.
\]
\end{proposition}

\begin{proof}
If \(n=1\), there is a single fiber of size \(q^\ell\).  Any two elements
outside a proposed class-resolving set have identical relation vectors, so
at most one element may be omitted.  Conversely, all but one element form
a class-resolving set.  Hence
\[
        \cdim(\Omega_1)=q^\ell-1,
\]
which is the stated formula when \(\qbinom11=1\).

Now assume \(n\geq2\).  For \(m=1\), the attenuated association scheme has
exactly three relation classes: the diagonal relation, the non-diagonal
same-fiber relation, and the different-fiber relation.  Equivalently, after
relabelling the two nontrivial relations, its relation code is the distance
code of the complete multipartite graph with parts \(F_U\),
\(U\in\Gr_q(W,1)\).  There are at least two fibers, and the corresponding
graph is
\[
        K_{\underbrace{q^\ell,\ldots,q^\ell}_{\qbinom n1\text{ parts}}}.
\]
Indeed, a resolving set can omit at most one vertex from each part, and
omitting one vertex from every part works.  This gives the standard formula
\[
        \mu(K_{m_1,\ldots,m_r})=\sum_{i=1}^r (m_i-1)
\]
for complete multipartite graphs 
\cite[Proposition~1]{BaileySmallDistanceRegular}.
Applying it with \(r=\qbinom n1\) and \(m_i=q^\ell\) gives
\[
        \cdim(\Omega_1)
        =
        \mu(K_{q^\ell,\ldots,q^\ell})
        =
        \sum_{U\in \Gr_q(W,1)} (q^\ell-1)
        =
        (q^\ell-1)\qbinom n1.
\]
\end{proof}

Proposition~\ref{obs:rank-one} demonstrates that the lower-bound constant
in Corollary~\ref{cor:attenuated-fixedrank} is asymptotically sharp in rank
one, including its leading constant.  Thus the entropy obstruction captures
the true asymptotic class dimension here, not only its order of magnitude.

\section{Dual polar graphs}\label{sec:dualpolar}

Bailey and Spiga obtained a general upper bound for dual polar graphs from
the real rank of the point--generator incidence matrix, and concluded by
asking for lower bounds in arbitrary diameter \cite{BaileySpiga}.  In this 
section we determine the correct order in the fixed-rank regime.  The 
lower bound uses opposition as the default distance, and the upper bound 
uses the line geometry of the dual polar space.

Here is the argument in brief.  Opposition is the typical distance, so one
landmark carries little information and the entropy bound forces many
landmarks.  For the upper bound, one random generator distinguishes a pair
too rarely.  A whole pencil is much less likely to fail, so a constant
number of independently chosen pencils resolves every pair.

\begin{theorem}[Main dual-polar theorem]\label{thm:dual-main}
Fix $d\geq2$ and a classical polar-space type with parameter $e>0$.  Put
\[
        A=de+\binom d2,\qquad
        \alpha=\min\{1,e\},\qquad
        k_{d,e}=\left\lfloor\frac{2A}{\alpha}\right\rfloor+1.
\]
As $q\to\infty$ through the admissible field orders for that type,
\[
        \left(\frac{A}{e}+o(1)\right)q^e
        \leq
        \mu(\Gamma(q,d,e))
        \leq
        k_{d,e}(q^e+1).
\]
Consequently,
\[
        \mu(\Gamma(q,d,e))=\Theta_{d,e}(q^e).
\]
\end{theorem}

\subsection{Polar spaces, generators, and distance}

A polar space is obtained by first equipping the ambient vector space with 
a
form and then retaining only the subspaces that are compatible with that
form.  We recall the resulting geometry before turning to metric 
dimension;
see \cite[Sec.~4]{BishnoiFiniteGeometry} for an introductory account and
\cite[Ch.~9]{BCN} and \cite[Sec.~3.1]{VanDamKoolenTanaka} for the 
classical
families and their associated distance-regular graphs.

Let $V$ be a finite-dimensional vector space over a finite field.  The
classical polar spaces considered here arise from one of the following:
\begin{itemize}[leftmargin=2em]
\item a nondegenerate alternating or Hermitian form $\beta$, in which case 
a
      subspace $X\leq V$ is \emph{totally isotropic} when
      $\beta(x,y)=0$ for all $x,y\in X$;
\item a nondegenerate quadratic form $Q$, in which case $X$ is
      \emph{totally singular} when $Q(x)=0$ for every $x\in X$.
\end{itemize}
We use \emph{singular} for either condition.  The points and lines of the
polar space are the singular subspaces of vector dimensions one and two,
with incidence given by containment.  For a singular point $p$, the 
notation
$p^\perp$ refers to the hyperplane orthogonal to $p$ under $\beta$, or 
under
the polar form associated with $Q$.
For a subspace $X$, write $\Pts(X)$ for its one-dimensional subspaces.

The \emph{rank} $d$ is the maximum vector-space dimension of a singular
subspace.  Witt decomposition implies that every maximal singular subspace
has dimension $d$; these maximal subspaces are called \emph{generators}
\cite{TaylorClassicalGroups}.
Thus a generator has projective dimension $d-1$, although throughout this
paper we use vector-space dimensions.  The adjective ``dual'' in dual 
polar
graph reflects the fact that the vertices are these maximal objects rather
than the points of the polar space.

We let $q$ denote the order of the ambient field.  Consequently, in the
Hermitian cases $q$ is required to be a square, so the displayed powers
$q^{1/2}$ and $q^{3/2}$ are integers.  The classical families are 
conveniently packaged by a parameter $e$, and we note that every singular 
$(d-1)$-subspace is contained in $q^e+1$ generators.  We call such a 
subspace a \emph{panel}, and call the $q^e+1$ generators through it its 
\emph{pencil}.  Equivalently, a rank-one residue has $q^e+1$ generators.  
The relevant types are listed below.

\begin{center}
\begin{tabular}{lll}
\toprule
Polar space & Defining form & Parameter $e$ \\
\midrule
$W(2d-1,q)$ & alternating on $\F_q^{2d}$ & $1$ \\
$Q(2d,q)$ & parabolic quadratic on $\F_q^{2d+1}$ & $1$ \\
$Q^+(2d-1,q)$ & hyperbolic quadratic on $\F_q^{2d}$ & $0$ \\
$Q^-(2d+1,q)$ & elliptic quadratic on $\F_q^{2d+2}$ & $2$ \\
$H(2d-1,q)$, $q$ square & Hermitian on $\F_q^{2d}$ & $1/2$ \\
$H(2d,q)$, $q$ square & Hermitian on $\F_q^{2d+1}$ & $3/2$ \\
\bottomrule
\end{tabular}
\end{center}

Let $\mathcal P$ be one of these polar spaces and let $\mathcal G_d$ 
denote
its set of generators.  The dual polar graph $\Gamma(q,d,e)$ has vertex 
set
$\mathcal G_d$, with two generators adjacent when they meet in a singular
subspace of codimension one.  More generally, the graph distance is
\begin{equation}\label{eq:dual-distance}
        d_\Gamma(X,Y)=d-\dim(X\cap Y).
\end{equation}
This is the standard distance formula for dual polar graphs
\cite[Sec.~9.4.1]{BCN}.  One edge can change the dimension of the
intersection with a fixed generator by at most one, while the extension
property of polar spaces allows one to replace the missing dimensions one
at a time.  The diameter is
therefore $d$.  Two generators at distance $d$ are disjoint and are said 
to
be \emph{opposite}.

In rank two, the polar space is a classical generalized quadrangle.  Its
generators are the singular lines, two vertices of the dual polar graph 
are
adjacent when the corresponding lines meet in a point, and opposite 
vertices
correspond to disjoint lines.  Higher rank is the same intersection 
geometry
with generators of dimension $d$.

The parameter $e$ controls all counts needed below.  Write
\[
        \qbinom r1=\frac{q^r-1}{q-1}.
\]
If $U$ is a fixed generator and $Y$ is at distance $i$ from $U$, the
intersection numbers of the dual polar graph are as follows
\cite[Sec.~9.4.1]{BCN}:
\begin{equation}\label{eq:dual-intersection-numbers}
        c_i=\qbinom i1
        \qquad (1\leq i\leq d),
        \qquad
        b_i=q^{e+i}\qbinom{d-i}{1}
        \qquad (0\leq i\leq d-1),
\end{equation}
where $c_i$ counts neighbors of $Y$ one step closer to $U$ and $b_i$ 
counts
neighbors one step farther away.  These formulas come from counting the
relevant hyperplanes of $Y$ and then extending them to generators in the
corresponding residues.  If
\[
        k_i=|\{Y\in\mathcal G_d:d_\Gamma(U,Y)=i\}|,
\]
then distance-regularity gives $k_0=1$ and
\[
        k_{i+1}=k_i\frac{b_i}{c_{i+1}}.
\]
Substituting \eqref{eq:dual-intersection-numbers} yields
\begin{equation}\label{eq:dual-sphere-size}
        k_i=q^{ie+\binom i2}\qbinom di,
        \qquad 0\leq i\leq d.
\end{equation}
The finite $q$-binomial theorem now gives the total number of generators:
\begin{equation}\label{eq:dual-generator-count}
        N_d:=|\mathcal G_d|
        =\sum_{i=0}^d k_i
        =\prod_{j=1}^d(q^{e+j-1}+1).
\end{equation}

The same product occurs in every residue.  If $R$ is a singular subspace 
of
vector dimension $d-t$, then the generators containing $R$ correspond to 
the
generators of the rank-$t$ polar space on $R^\perp/R$, which has the same
parameter $e$.  We therefore put
\begin{equation}\label{eq:dual-residue-count}
        N_t=\prod_{j=1}^t(q^{e+j-1}+1),
        \qquad N_0=1.
\end{equation}
In particular, a fixed singular point lies in $N_{d-1}$ generators and a
fixed singular line lies in $N_{d-2}$ generators.

The hyperbolic family $Q^+(2d-1,q)$ has $e=0$ and behaves differently in 
the
asymptotic regime considered here.  By \eqref{eq:dual-sphere-size} and
\eqref{eq:dual-generator-count}, the probability that two random 
generators
are opposite tends to $1/2$, rather than to $1$.  Thus maximum distance is
not an overwhelmingly likely response, and the default-distance argument
below does not produce a growing $q^e$ scale.  We therefore assume $e>0$ 
for
the remainder of the section.

\subsection{Opposition and the entropy lower bound}

For $d=1$, the graph is complete on $q^e+1$ vertices, and hence has metric
dimension $q^e$.  We henceforth take $d\geq2$.

Fix a generator $U$ and let $X$ be uniformly distributed on $\mathcal 
G_d$.
The maximum-distance response $d_\Gamma(U,X)=d$ is the event that $X$ is
opposite $U$.  From \eqref{eq:dual-sphere-size} and
\eqref{eq:dual-generator-count},
\begin{equation}\label{eq:dual-opposition-probability}
        \Prob(X\cap U=0)
        =\frac{k_d}{N_d}
        =\prod_{j=1}^d(1+q^{-(e+j-1)})^{-1}.
\end{equation}
The first factor in this product controls the failure of opposition.  In
particular, a fixed landmark has a non-default response on only about a
$q^{-e}$-fraction of all generators.

\begin{theorem}\label{thm:dual-lower}
For fixed $d\geq2$, fixed classical type with parameter $e>0$, and
$q\to\infty$ through the admissible prime powers for that type,
\[
        \mu(\Gamma(q,d,e))
        \geq
        \left(\frac{de+\binom d2}{e}+o(1)\right)q^e.
\]
\end{theorem}

\begin{proof}
By Witt's extension theorem, the full isometry group of the form acts
transitively on generators \cite{TaylorClassicalGroups}, so
Theorem~\ref{thm:vt-entropy} applies.  Let
\[
        \eps_q=\Prob(d_\Gamma(U,X)\neq d)
        =\Prob(X\cap U\neq0).
\]
By \eqref{eq:dual-opposition-probability},
\[
        \eps_q
        =1-\prod_{j=1}^d(1+q^{-(e+j-1)})^{-1}.
\]
Expanding this finite product, with $d$ and $e$ fixed, gives
\begin{equation}\label{eq:dual-nonopposition-asymptotic}
        \eps_q
        =q^{-e}+O_{d,e}(q^{-e-1}+q^{-2e}).
\end{equation}
Thus
\[
        \log(1/\eps_q)=e\log q+o(\log q).
\]
In particular, $\eps_q\leq1/2$ for all sufficiently large admissible $q$.
The distance response has the fixed alphabet $\{0,1,\ldots,d\}$, so
Lemma~\ref{lem:default} gives
\[
        H(d_\Gamma(U,X))
        \leq
        \eps_q\log(1/\eps_q)+O_d(\eps_q)
        =(e+o(1))q^{-e}\log q.
\]
On the other hand, \eqref{eq:dual-generator-count} gives
\[
        \log N_d
        =\sum_{j=1}^d\log(q^{e+j-1}+1)
        =\left(de+\binom d2+o(1)\right)\log q.
\]
Substituting into Theorem~\ref{thm:vt-entropy} yields
\[
        \mu(\Gamma(q,d,e))
        \geq
        \frac{\left(de+\binom d2+o(1)\right)\log q}
        {(e+o(1))q^{-e}\log q}
        =
        \left(\frac{de+\binom d2}{e}+o(1)\right)q^e.
\]
\end{proof}

There is a useful geometric reading of the proof.  For a generator $W$, 
let
\[
        \mathcal N(W)=\{Z\in\mathcal G_d:Z\cap W\neq0\}
\]
be its non-opposition neighborhood.  If two generators lie outside
$\bigcup_{s\in S}\mathcal N(s)$, then both are opposite every landmark in
$S$ and hence have the same constant distance vector $(d)_{s\in S}$.
Therefore a resolving set must cover all but at most one generator by 
these
non-opposition neighborhoods.  Since each neighborhood has relative size
$(1+o(1))q^{-e}$, this already explains the scale $q^e$.  The entropy
calculation also accounts for the full information content $\log N_d$ and
produces the leading constant in Theorem~\ref{thm:dual-lower}.

\subsection{Simultaneous non-opposition}

To obtain an upper bound by random sampling, it is not enough to know the
response distribution around one fixed generator.  We need a lower bound 
on
the probability that a random landmark distinguishes an arbitrary pair
$U\neq V$, uniformly over the possible values of $\dim(U\cap V)$.  The 
main
point is that simultaneous non-opposition to both $U$ and $V$ is of 
smaller
order than non-opposition to either one separately.

From \eqref{eq:dual-nonopposition-asymptotic}, uniformly in $W\in\mathcal 
G_d$,
\begin{equation}\label{eq:dual-one-nonopposition}
        \Prob(Z\in\mathcal N(W))=(1+o(1))q^{-e},
        \qquad Z\sim\Unif(\mathcal G_d).
\end{equation}

\begin{lemma}\label{lem:dual-both}
Assume $d\geq2$ and $e>0$ are fixed.  If $U,V\in\mathcal G_d$ are distinct
and $Z\sim\Unif(\mathcal G_d)$, then, uniformly over all $U\neq V$,
\[
        \Prob(Z\cap U\neq0,\ Z\cap V\neq0)
        =O_{d,e}(q^{-e-1}+q^{-2e})
        =o(q^{-e}).
\]
\end{lemma}

\begin{proof}
Suppose that $Z$ meets both $U$ and $V$.  Choose projective points
$p\in\Pts(Z\cap U)$ and $p'\in\Pts(Z\cap V)$.  If $p=p'$, then $Z$ 
contains
a point of $U\cap V$.  If $p\neq p'$, then $p$ and $p'$ lie in the 
singular
generator $Z$, so they span a singular line.  We bound these two cases
separately.

Since $U\neq V$, the vector dimension of $U\cap V$ is at most $d-1$.
Consequently,
\[
        |\Pts(U\cap V)|\leq\qbinom{d-1}{1}=O_d(q^{d-2}).
\]
A fixed singular point lies in $N_{d-1}$ generators.  By
\eqref{eq:dual-residue-count}, the probability that $Z$ contains at least 
one
common point of $U$ and $V$ is therefore at most
\[
        O_d(q^{d-2})\frac{N_{d-1}}{N_d}
        =O_d(q^{d-2})\frac{1}{q^{e+d-1}+1}
        =O_{d,e}(q^{-e-1}).
\]

It remains to count ordered pairs $(p,p')\in\Pts(U)\times\Pts(V)$ that 
span
a singular line.  If $p\in U\cap V$, then there are $O_d(q^{d-2})$ choices
for $p$.  We may overcount by allowing all $O_d(q^{d-1})$ choices for
$p'\in\Pts(V)$, contributing
$O_d(q^{2d-3})$ pairs.  Now suppose $p\in U\setminus V$.  The singular 
points
of $V$ collinear with $p$ are precisely the points of $V\cap p^\perp$.
This is a proper subspace of $V$: if $V\subseteq p^\perp$, then
$\langle V,p\rangle$ would be a singular subspace of vector dimension
$d+1$, contradicting the maximality of $V$.  Hence
\[
        |\Pts(V\cap p^\perp)|=O_d(q^{d-2}).
\]
There are $O_d(q^{d-1})$ choices for $p$, so the number of ordered 
collinear
pairs is again $O_d(q^{2d-3})$, uniformly in $U$ and $V$.

Each singular line lies in $N_{d-2}$ generators.  A union bound over the
ordered pairs therefore gives
\[
        O_d(q^{2d-3})\frac{N_{d-2}}{N_d}
        =O_d(q^{2d-3})
          \frac{1}{(q^{e+d-2}+1)(q^{e+d-1}+1)}
        =O_{d,e}(q^{-2e}).
\]
Adding the common-point and singular-line contributions proves the lemma.
\end{proof}

The estimate in Lemma~\ref{lem:dual-both} is stronger than what is needed
merely to distinguish one fixed pair.  It says that the generators meeting
both members of a pair form a lower-order exceptional set.  The upper 
bound
will amplify this fact by grouping generators into pencils.

\subsection{Panel pencils and the upper bound}

Viewed as a point--line geometry whose points are generators and whose 
lines
are pencils through panels, a dual polar space is a near $2d$-gon
\cite{ShultYanushka,CameronDualPolar}.  In particular, for every generator
$U$ and every pencil $\mathcal C(H)$ there is a unique member of the 
pencil
nearest to $U$; this is the \emph{gate} of $U$ on the pencil.  The pencils
are precisely the lines of this near polygon and, in the dual polar graph,
are maximal cliques; see also \cite[Sec.~9.4]{BCN}.  We retain a short
self-contained proof of the weak gate property needed below.

For a panel $H$, write
\[
        \mathcal C(H)=\{W\in\mathcal G_d:H\leq W\}
\]
for its pencil.  Thus $|\mathcal C(H)|=q^e+1$.  The advantage of taking 
the
whole pencil is that it is guaranteed to contain a generator non-opposite 
to
any prescribed target.

\begin{lemma}\label{lem:dual-pencil-meets}
Let $H$ be a panel and let $U$ be a generator.  Then some
$W\in\mathcal C(H)$ satisfies $W\cap U\neq0$.
\end{lemma}

\begin{proof}
Let $H^\perp$ denote the subspace orthogonal to $H$ under the defining
sesquilinear form, or under the polar form associated with the quadratic
form.  Since $\dim H=d-1$, orthogonality imposes at most $d-1$ independent
linear conditions.  Hence
\[
        \dim H^\perp\geq \dim V-(d-1),
\]
and therefore
\[
        \dim(U\cap H^\perp)
        \geq \dim U+\dim H^\perp-\dim V
        \geq1.
\]
Choose a singular point $p\leq U\cap H^\perp$.  If $p\leq H$, then every
generator in $\mathcal C(H)$ meets $U$.  Otherwise $H+p$ is singular of
vector dimension $d$, and hence is itself a generator.  Taking
$W=H+p$ proves the claim.
\end{proof}

Lemma~\ref{lem:dual-pencil-meets} is the weak existence part of the gate
property.  In rank two, it is the familiar generalized-quadrangle axiom:
given a point and a nonincident singular line, there is a unique point of
the line collinear with the given point.  Here a panel is a singular point,
and its pencil is the set of singular lines through it.  Thus sampling a
pencil means taking all lines through one selected point.

For distinct generators $U,V$, call a panel $H$ \emph{bad for $(U,V)$} if
no generator in its pencil distinguishes them.  Equivalently,
\[
        d_\Gamma(W,U)=d_\Gamma(W,V)
        \qquad\text{for every }W\in\mathcal C(H).
\]

\begin{remark}\label{rem:dual-gates}
Let $g_H(U)$ denote the gate of $U$ on $\mathcal C(H)$.  If
$W\neq g_H(U)$ lies in the same pencil, then
\[
        d_\Gamma(U,W)=d_\Gamma(U,g_H(U))+1,
\]
because $W$ is adjacent to the gate and the gate is the unique nearest
member.  It follows that $H$ is bad for $(U,V)$ if and only if
\[
        g_H(U)=g_H(V)
        \quad\text{and}\quad
        d_\Gamma(U,g_H(U))=d_\Gamma(V,g_H(V)).
\]
We do not need this exact characterization for the order bound, but it
identifies the structure that would have to be counted to improve the 
upper
constant.
\end{remark}

The next lemma is the amplification step.  A random generator meets both
$U$ and $V$ with probability $o(q^{-e})$; after passing to pencils, this
becomes a polynomially small probability that the entire pencil is bad.

\begin{lemma}\label{lem:dual-bad-panel}
Let $\mathcal P_{d-1}$ be the set of panels and put
\[
        \alpha=\min\{1,e\}.
\]
For fixed $d\geq2$ and $e>0$, uniformly over distinct generators $U,V$,
\[
 \frac{|\{H\in\mathcal P_{d-1}:H\text{ is bad for }(U,V)\}|}
      {|\mathcal P_{d-1}|}
 =O_{d,e}(q^{-\alpha}).
\]
\end{lemma}

\begin{proof}
Let $H$ be bad for $(U,V)$.  By Lemma~\ref{lem:dual-pencil-meets}, some
$W\in\mathcal C(H)$ meets $U$.  Then $d_\Gamma(W,U)<d$.  Since the pencil
fails to distinguish $U$ and $V$, we also have
$d_\Gamma(W,V)=d_\Gamma(W,U)<d$, and hence $W$ meets $V$ as well.
Consequently every bad panel is contained in at least one generator 
meeting
both $U$ and $V$.

Each generator contains $\qbinom d1$ panels, so
\[
\begin{aligned}
 |\{H:H\text{ is bad for }(U,V)\}|
 &\leq
 \qbinom d1
 |\{W\in\mathcal G_d:W\cap U\neq0,\ W\cap V\neq0\}|.
\end{aligned}
\]
On the other hand, double-counting incident pairs $(H,W)$ with $H\leq W$
gives
\[
        |\mathcal P_{d-1}|(q^e+1)
        =N_d\qbinom d1.
\]
Dividing the preceding inequalities and applying
Lemma~\ref{lem:dual-both}, we obtain
\[
\begin{aligned}
 \frac{|\{H:H\text{ is bad for }(U,V)\}|}{|\mathcal P_{d-1}|}
 &\leq
 (q^e+1)
 \Prob_W(W\cap U\neq0,\ W\cap V\neq0)\\
 &=O_{d,e}(q^{-1}+q^{-e})
 =O_{d,e}(q^{-\alpha}).
\end{aligned}
\]
The estimate is uniform because Lemma~\ref{lem:dual-both} is uniform.
Here the pair $U,V$ may vary with $q$; only $d$, $e$, and the classical
type are fixed.
\end{proof}

One random generator has only a $\Theta(q^{-e})$ chance to distinguish a
fixed pair, which leads to a logarithmic covering cost when generators are
sampled independently.  A random pencil behaves differently: its failure
probability is already $O(q^{-\alpha})$.  Since the number of generator 
pairs
is only polynomial in $q$ for fixed rank, a constant number of pencils is
enough.

\begin{theorem}\label{thm:dual-upper}
Fix $d\geq2$ and a classical type with parameter $e>0$.  Put
\[
        A=de+\binom d2,
        \qquad
        \alpha=\min\{1,e\},
\]
and let $k$ be any integer satisfying
\[
        k>\frac{2A}{\alpha}.
\]
Then, for all sufficiently large admissible prime powers $q$,
\[
        \mu(\Gamma(q,d,e))\leq k(q^e+1).
\]
In particular,
\[
        \mu(\Gamma(q,d,e))=O_{d,e}(q^e).
\]
\end{theorem}

\begin{proof}
Choose panels $H_1,\ldots,H_k$ independently and uniformly from
$\mathcal P_{d-1}$, and set
\[
        S=\bigcup_{i=1}^k\mathcal C(H_i).
\]
For a fixed unordered pair $\{U,V\}$, Lemma~\ref{lem:dual-bad-panel}
gives a constant $C=C(d,e)$ such that
\[
        \Prob(H_i\text{ is bad for }(U,V))\leq Cq^{-\alpha}
\]
for all sufficiently large $q$.  Independence therefore gives
\[
        \Prob(S\text{ fails to distinguish }U,V)
        \leq C^kq^{-k\alpha}.
\]

By \eqref{eq:dual-generator-count},
\[
        N_d=\Theta_{d,e}(q^A),
\]
so the number of unordered pairs of generators is $O_{d,e}(q^{2A})$.
The union bound yields
\[
\begin{aligned}
        \Prob(S\text{ fails to resolve }\Gamma(q,d,e))
        &\leq O_{d,e}(q^{2A-k\alpha}).
\end{aligned}
\]
Since $k\alpha>2A$, this probability tends to zero.  Thus for all
sufficiently large $q$ there is a choice of the panels for which $S$
resolves the graph.  Finally,
\[
        |S|\leq\sum_{i=1}^k|\mathcal C(H_i)|
        =k(q^e+1),
\]
which proves the theorem.
\end{proof}

\begin{proof}[Proof of Theorem~\ref{thm:dual-main}]
The lower bound is Theorem~\ref{thm:dual-lower}.  For the upper bound, 
take
\[
        k=k_{d,e}
        =\left\lfloor
          \frac{2\left(de+\binom d2\right)}{\min\{1,e\}}
          \right\rfloor+1
\]
in Theorem~\ref{thm:dual-upper}.  By construction,
$k\min\{1,e\}>2\left(de+\binom d2\right)$, so that theorem applies for all
sufficiently large admissible $q$.  The two estimates give the asserted
$\Theta_{d,e}(q^e)$ conclusion.
\end{proof}

For the symplectic and parabolic orthogonal families, $e=1$ and
$A=d+\binom d2=d(d+1)/2$.  Hence one may take
\[
        k_{d,1}=d^2+d+1,
\]
and Theorem~\ref{thm:dual-main} becomes
\[
        \left(\frac{d(d+1)}2+o(1)\right)q
        \leq \mu(\Gamma(q,d,1))
        \leq (d^2+d+1)(q+1).
\]
More generally, the ratio of the leading upper coefficient $k_{d,e}$ to
the entropy lower coefficient $A/e$ is at most
\[
        \frac{2e}{\min\{1,e\}}+\frac{e}{A},
\]
where the second term accounts for rounding $k_{d,e}$ to an integer.  Thus
the remaining gap is a bounded-factor problem, not a difference of
asymptotic scale.

\begin{remark}
The same pair estimate, instead applied to independently sampled 
generators, gives
an $O_{d,e}(q^e\log q)$ resolving set: a fixed pair is distinguished with
probability $\Theta(q^{-e})$.  Indeed, meeting exactly one member of the
pair is enough, and Lemma~\ref{lem:dual-both} shows that meeting both is
lower order.  A union bound over all pairs then costs a factor of $\log q$.
A pencil packages $q^e+1$ highly correlated
generators
whose common failure probability is already polynomially small.  This is
the step that removes the logarithm.
\end{remark}

The resulting upper bound is also polynomially smaller than the
incidence-rank bound of Bailey and Spiga.  Their bound is
\[
        \frac{(q^{d+e-1}+1)
        (q^{d+e-1}-q^{e-1}+q-1)}
        {(q^{e-1}+1)(q-1)}.
\]
For example, in the symplectic and parabolic orthogonal cases $e=1$, this
is $\Theta_d(q^{2d-1})$, whereas Theorem~\ref{thm:dual-upper} gives
$O_d(q)$.  The symplectic subfamily had previously been treated separately
by Guo, Wang, and Li \cite{GuoWangLiSymplectic}; the pencil argument applies
uniformly to all classical types with $e>0$.  In rank two, specialized 
generalized-quadrangle constructions can give
better explicit upper constants \cite[Sec.~4]{BaileySpiga}, but the 
theorem
above supplies a common lower bound and extends without change to every
fixed rank.

\section{Further questions}\label{sec:questions}

The examples above suggest that metric and class dimension in
finite-geometric schemes are often controlled by the first exceptional
stratum from a fixed landmark.  The dual polar result also leaves several
more geometric problems.
\begin{enumerate}[label=(\arabic*), leftmargin=2.2em]
\item Determine the leading constant of
\[
        q^{-e}\mu(\Gamma(q,d,e))
\]
for fixed $d$ and $e>0$.  The entropy lower constant is
$(de+\binom d2)/e$, while the panel-pencil construction differs from it by
only a bounded factor.  A natural first step is to count bad panels
separately for the different values of $\dim(U\cap V)$, or equivalently to
use the gate characterization in Remark~\ref{rem:dual-gates}.

\item Understand dual polar metric dimension when the rank $d$ grows with
$q$.  The present estimates are uniform only after $d$ and the classical
type are fixed.

\item Treat the hyperbolic family $e=0$.  Opposition is no longer a
high-probability default relation, so a different distribution or a
different geometric bundle is required.  The related halved dual polar,
Ustimenko, and Hemmeter graphs are natural test cases.

\item Determine whether the leading entropy constant for the bilinear 
forms
graph is the true leading constant of $\mu(H_q(n,d))$, and find exact or
leading-constant estimates for attenuated-space class dimension beyond 
rank
one.

\item Strengthen the entropy method beyond singleton subadditivity.
Pairwise relation distributions in an association scheme are encoded by
intersection numbers, so Shearer-type inequalities or nonuniform target
distributions may recover correlations that the present argument discards.
\end{enumerate}

\section*{Acknowledgments}

The author thanks Jesse Geneson for many valuable discussions and helpful feedback.
Codex with GPT-5.6 was used to
explore and critique proofs and assist with exposition and revision.

\end{document}